\documentclass[11pt]{article}
\usepackage{amsthm,amsfonts, amsbsy, amssymb,amsmath,graphicx}
\usepackage{graphics}
\usepackage[english]{babel}
\usepackage{graphicx}
\usepackage{lineno}
\usepackage{enumerate}
\usepackage{tikz}
\usepackage{hyperref}
\usepackage{color}
\usepackage{tikz}

\hypersetup{colorlinks=true}

\hypersetup{colorlinks=true, linkcolor=blue, citecolor=blue,urlcolor=blue}

\input epsf
\usepackage{epic,latexsym,amssymb}
\usepackage{tikz}

\usepackage{tikz,epsf}

\usepackage{amsfonts}
\usepackage{amscd}
\usepackage{amsmath}
\usepackage{graphicx}
\usepackage{color}
\usepackage{caption,subcaption}

\newtheorem{theorem}{Theorem}
\newtheorem{observation}[theorem]{Observation}
\newtheorem{proposition}[theorem]{Proposition}
\newtheorem{lemma}[theorem]{Lemma}

\newtheorem{corollary}[theorem]{Corollary}

\tikzstyle{vertex}=[circle, draw, inner sep=0pt, minimum size=6pt]

\newcommand{\QEDmark}{\mbox{\textsc{qed}}}
\newcommand{\proofStarter}[1]{\textsc{#1} }

\begin{document}

\title{ Restrained double Roman domination of a graph}
\author{{\small  Doost Ali Mojdeh$^{a}$\thanks{Corresponding author} ,  Iman  Masoumi$^{b}$ Lutz Volkmann$^{c}$}\\ \\ {\small $^{a}$Department of
Mathematics, Faculty of Mathematical Sciences}\\{\small University of Mazandaran, Babolsar, Iran}\\{\small email: damojdeh@umz.ac.ir}\\ \\
{\small $^{b}$Department of Mathematics, University of Tafresh}\\{\small Tafresh, Iran}\\ {\small email: i\_masoumi@yahoo.com}\\ \\
{\small $^{c}$Lehrstuhl II f\"{u}r Mathematik, RWTH Aachen University}\\ {\small 52056 Aachen, Germany}\\ {\small email: volkm@math2.rwth-aachen.de}  }
\date{}
\maketitle

\begin{abstract}
For a graph $G=(V,E)$, a restrained double Roman dominating function is a function $f:V\rightarrow\{0,1,2,3\}$ having the property that if $f(v)=0$,
then the vertex $v$ must have at least two neighbors
assigned $2$ under $f$ or one neighbor $w$ with $f(w)=3$, and if $f(v)=1$, then the vertex $v$ must have at least one neighbor $w$ with $f(w)\geq2$,
and at the same time, the subgraph $G[V_0]$ which
includes vertices with zero labels has no isolated vertex. The weight of a restrained double Roman dominating function $f$ is the sum
$f(V)=\sum_{v\in V}f(v)$, and the minimum weight of a restrained
double Roman dominating function on $G$ is the restrained double Roman domination number of $G$. We initiate the study of restrained double Roman
domination with proving that the problem of computing
this parameter is $NP$-hard. Then we present an upper bound on the restrained double Roman domination number of a connected graph $G$ in terms of the
order of $G$ and characterize the graphs attaining this bound.
We study the restrained double Roman domination versus the restrained Roman domination. Finally,   we characterized all trees $T$ attaining the exhibited
bound.
\end{abstract}

\textbf{2010 Mathematical Subject Classification:} 05C69

\textbf{Keywords}: Domination, restrained Roman domination, restrained double Roman domination.

\section{Introduction}
Throughout this paper, we consider $G$ as a finite simple graph with vertex set $V=V(G)$ and edge set $E=E(G)$. We use \cite{west} as a
reference for terminology and notation which are not explicitly defined
here. The open neighborhood of a vertex $v$ is denoted by $N(v)$, and its closed neighborhood is
$N[v]=N(v)\cup \{v\}$. The minimum and maximum degrees of $G$ are denoted by $\delta(G)$ and $\Delta(G)$,
respectively. Given subsets $A,B \subseteq V(G)$, by $[A,B]$ we mean the set of all edges with one end
point in $A$ and the other in $B$. For a given subset $S\subseteq V(G)$, by $G[S]$ we represent the subgraph
induced by $S$ in $G$. A tree $T$ is a double star if it contains exactly two vertices that are not leaves.
A double star with $p$ and $q$ leaves attached to each support vertex, respectively, is denoted by
$S_{p,q}$. A wounded spider is a tree obtained from subdividing at most $n-1$ edges of a star $K_{1,n}$.
A wounded spider obtained by subdividing $t \le n-1$  edges of $K_{1,n}$, is denoted by $ws(1,n, t)$.\\

A set $S\subseteq V(G)$ is called a dominating set if every vertex not in $S$ has a neighbor in $S$. The domination number
$\gamma(G)$ of $G$ is the minimum cardinality among all dominating sets of $G$. A
restrained dominating set ($RD$ set) in a graph $G$ is a dominating set $S$ in $G$ for which every vertex
in $V(G)-S$ is adjacent to another vertex in $V(G)-S$. The restrained domination number ($RD$
number) of $G$, denoted by
$\gamma_r(G)$, is the smallest cardinality of an $RD$ set of $G$. This concept was
formally introduced in \cite{domke} (albeit, it was indirectly introduced in \cite{hattingh, haynes}).

The variants of restrained domination have been already worked. For instance,  a total restrained domination of a graph $G$ is an $RD$ set of $G$
for which  the subgraph induced by the dominating set of $G$
has no isolated vertex,  which can be referred to the \cite{chen}. Secure restrained dominating set $(SRDS)$ which is a set
$S \subseteq V(G)$ for which $S$ is restrained dominating and for all $u \in V\setminus S$ there exists $v \in S\cap N(u)$ such that
$(S\setminus \{v\})\cup \{u\}$ is restrained dominating set \cite{roushini}.

The restrained Roman dominating function is a Roman dominating function $f: V(G) \to \{0,1,2\}$ such that the subgraph induced by the set
$\{v\in V(G): f(v)=0\}$ has no isolated vertex, \cite{roushini1}.
The restrained Italian dominating function ($RIDF$) is an Italian dominating function $f: V(G) \to \{0,1,2\}$ such that the subgraph induced by the
set $\{v\in V(G): f(v)=0\}$ has no isolated vertex, \cite{samadi}.

These results motivates us to consider a double Roman dominating function $f$ for which the subgraph induced by $V_0^f$ has no isolated vertex,
 which is the concept that we stand on it as new  parameter namely restrained double Roman domination and will be investigated in this paper

 Beeler \emph{et al}. (2016) \cite{bhh} introduced the concept of double Roman domination of a graph.\\
 If $ f:V(G)\rightarrow \{0,1,2,3\}$ is a function, then let $(V_0,V_1,V_2,V_3)$ be the ordered partition of $V(G)$ induced by $f$, where
 $V_i=\{v\in V(G):f(v)=i\}$
for $i=0,1,2,3$. There is a 1-1 correspondence between the function $f$ and the ordered partition $(V_0,V_1,V_2,V_3)$. So we will write
$f=(V_0,V_1,V_2,V_3)$.
A double Roman dominating function   (DRD function  for short) of a graph $G$ is a function $ f:V(G)\rightarrow \{0,1,2,3\}$ for which the following
conditions are satisfied.
\begin{itemize}
  \item[(a)] If $f(v)=0$, then the vertex $v$ must have at least two neighbors in $V_2$ or one neighbor in $V_3$.
  \item[(b)] If $f(v)=1$ , then the vertex $v$ must have at least one neighbor in $V_2\cup V_3$.
\end{itemize}
This parameter was also studied in \cite{al}, \cite{jr}, \cite{mojdeh} and \cite{zljs}.

 Accordingly,  a restrained double Roman dominating function ({$RDRD$} function for short) is a double Roman dominating function
 $f:V\rightarrow\{0,1,2,3\}$ having the property that:
  the subgraph induced by $V_0$ (the vertices with zero labels under $f$) $G[V_0]$  has no isolated vertex. The restrained double Roman domination
  number ($RDRD$ number) $\gamma_{rdR}(G)$ is the minimum weight  of an $RDRD$ function $f$ of $G$. For the sake of convenience, an $RDRD$
  function $f$ of a graph $G$ with weight $\gamma_{rdR}(G)$ is called a $\gamma_{rdR}(G)$-function.

  This paper is organized as follows.  We prove that the restrained double Roman domination problem  is $NP$-hard even for general graphs. Then,
  we present an upper bound on the restrained double Roman domination number of a connected graph $G$ in terms of the order of $G$ and characterize
  the graphs attaining this bound.
We study the restrained double Roman domination versus the restrained Roman domination. Finally, we characterize trees $T$ by the given restrained
double Roman domination number of $T$.

\section{Complexity and computational issues}
We consider the problem of deciding whether a graph $G$ has an $RDRD$ function of weight at most
a given integer. That is stated in the following decision problem.\\
 We shall prove the $NP$-completeness by reducing the following
vertex cover decision problem, which is known to be $NP$-complete.
\vspace{3mm}

\framebox{
\parbox{1\linewidth}{
VERTEX COVER DECISION PROBLEM
INSTANCE: A graph $G = (V,E)$ and a positive integer $p \le |V (G)|$.
QUESTION: Does there exist a subset $C \subseteq V (G)$ of size at most $p$ such that
for each edge $xy \in  E(G)$ we have $x \in C$ or $y \in C$?}}
\vspace{3mm}

\begin{theorem}
 \emph{(Karp \cite{karp} )}\label{the-karp} Vertex cover decision problem is $NP$-complete for general
graphs.
\end{theorem}
\vspace{3mm}

\framebox{
\parbox{1\linewidth}{
 RISTRAINED DOUBLE ROMAN DOMINATION problem ($RDRD$ problem)\\
INSTANCE: A graph $G$ and an integer $p\leq |V(G)|$.\\
QUESTION: Is there an $RDRD$ function $f$ for $G$ of weight at most $p$?}}\\
\vspace{3mm}

\begin{theorem}\label{the-NP}
The restrained double Roman domination problem is $NP$-complete for general
graphs.
\end{theorem}

\begin{proof}
We transform the vertex cover decision problem for general graphs to the
restrained double Roman domination decision problem for general graphs. For a given
graph $G = (V(G), E(G))$, let $ m= 3|V (G)| + 4$ and construct a graph $H = (V(H),E(H))$ as
follows. Let $V(H) = \{x_i : 1 \le i \le m\} \cup \{y\} \cup  V (G) \cup \{u_{j_i} : 1 \le i \le m\ \mbox{for\ each}\ e_j \in E(G)\}$, and let
$$E(H)=\{x_ix_{i+1}: (\mbox{mod}\ m)\ 1\le i\le m\}$$ $$\ \ \ \cup \{x_iy: 1\le i\le m\}  \cup \{vy: v\in V(G)\}$$ $$\ \ \
\cup \{vu_{j_i}: v\ \mbox{is\ the\ vertex\ of\  edge}\ e_j\in E(G)\ \mbox{and}\  1\le i \le m \}$$ $$\ \ \ \cup \{u_{j_i}u_{j_{(i+1)}}\
(\mbox{mod}\ m) : 1 \le i \le m\}.$$

Figure 1 shows the graph $H$ obtained from $G = P_4=a_1a_2a_3a_4$ by the above procedure. Note that, since $m= 3|V (G)| + 4=16$ for this example and $G$

\begin{figure}[h]
\tikzstyle{every node}=[circle, draw, fill=black!, inner sep=0pt,minimum width=.16cm]
\begin{center}
\begin{tikzpicture}[thick,scale=.6]
  \draw(0,0) { 
 +(-1,1) node{}
    +(-1,-1) node{}
 +(-1,-3) node{}
    +(-1,-5) node{}
    +(-1,1) -- +(-1,-5)

+(-.3,1.) node[rectangle, draw=white!0, fill=white!100]{ $\small{a_1}$}
+(-.3,-1) node[rectangle, draw=white!0, fill=white!100]{ $\small{a_2}$}
+(-.3,-3) node[rectangle, draw=white!0, fill=white!100]{ $\small{a_3}$}
+(-.3,-5) node[rectangle, draw=white!0, fill=white!100]{ $\small{a_4}$}
+(-1.1,-5.7) node[rectangle, draw=white!0, fill=white!100]{ $\small{G}$}

+(4,2) node{}
    +(4,0) node{}

+(4,-2.25) node[rectangle, draw=white!0, fill=white!100]{ ${\textbf{.}}$}
+(4,-2.5) node[rectangle, draw=white!0, fill=white!100]{ ${\textbf{.}}$}
+(4,-2.75) node[rectangle, draw=white!0, fill=white!100]{ ${\textbf{.}}$}
 +(4,-5) node{}
    +(4,-7) node{}
    +(4,2) -- +(4,-1.7)

 +(4,-3.3) -- +(4,-7)

 +(8,-2.5) node{}
+(8.,-1.7) node[rectangle, draw=white!0, fill=white!100]{ $\small{y}$}

+(4.8,1.8) node[rectangle, draw=white!0, fill=white!100]{ $\small{x_1}$}
+(4.8,.0) node[rectangle, draw=white!0, fill=white!100]{ $\small{x_2}$}
+(4.8,-5.2) node[rectangle, draw=white!0, fill=white!100]{ $\small{x_{15}}$}
+(4.8,-7) node[rectangle, draw=white!0, fill=white!100]{ $\small{x_{16}}$}

 +(8,-2.5) -- +(4,2)
 +(8,-2.5) -- +(4,0)
 +(8,-2.5) -- +(4,-5)
 +(8,-2.5) -- +(4,-7)
 +(8,-2.5) -- +(7,-2.75)
 +(8,-2.5) -- +(7,-2.25)

    +(12,3.5) node{}
 +(12.5,1.5) node{}
    +(12,-.5) node{}
  +(12.5,-2.5) node{}
    +(12,-4.5) node{}
    +(12.5,-6.5) node{}
 +(12,-8.5) node{}

 +(12.5,1.5) -- +(12,3.5)
  +(12.5,1.5) -- +(12,-.5)

   +(12.5,-2.5) -- +(12,-.5)
    +(12.5,-2.5) -- +(12,-4.5)
     +(12.5,-6.5) -- +(12,-4.5)
     +(12.5,-6.5) -- +(12,-8.5)

 +(8,-2.5) -- +(12,3.5)
    +(8,-2.5) -- +(12,-.5)   
    +(8,-2.5) -- +(12,-4.5)   
    +(8,-2.5) -- +(12,-8.5)

+(11.5,3.5) node[rectangle, draw=white!0, fill=white!100]{ $\small{a_1}$}
+(11.5,-.2) node[rectangle, draw=white!0, fill=white!100]{ $\small{a_2}$}
+(11.5,-4.8) node[rectangle, draw=white!0, fill=white!100]{ $\small{a_3}$}
+(11.5,-8.5) node[rectangle, draw=white!0, fill=white!100]{ $\small{a_4}$}

+(11.7,1.5) node[rectangle, draw=white!0, fill=white!100]{ $\small{u_{1_1}}$}
+(11.7,-2.5) node[rectangle, draw=white!0, fill=white!100]{ $\small{u_{2_1}}$}
+(11.7,-6.5) node[rectangle, draw=white!0, fill=white!100]{ $\small{u_{3_1}}$}

+(14.7,2) node[rectangle, draw=white!0, fill=white!100]{ $\small{u_{1_2}}$}
+(14.7,-2) node[rectangle, draw=white!0, fill=white!100]{ $\small{u_{2_2}}$}
+(14.7,-6) node[rectangle, draw=white!0, fill=white!100]{ $\small{u_{3_2}}$}

+(19.4,2.5) node[rectangle, draw=white!0, fill=white!100]{ $\small{u_{1_{15}}}$}
+(19.4,-1.5) node[rectangle, draw=white!0, fill=white!100]{ $\small{u_{2_{15}}}$}
+(19.4,-5.5) node[rectangle, draw=white!0, fill=white!100]{ $\small{u_{3_{15}}}$}

+(22.3,1.5) node[rectangle, draw=white!0, fill=white!100]{ $\small{u_{1_{16}}}$}
+(22.3,-2.5) node[rectangle, draw=white!0, fill=white!100]{ $\small{u_{2_{16}}}$}
+(22.3,-6.5) node[rectangle, draw=white!0, fill=white!100]{ $\small{u_{3_{16}}}$}

+(12,3.5) -- +(14.5,1.5)
 +(12,3.5) -- +(19.5,1.5)
 +(12,3.5) -- +(21.5,1.5)

 +(12,-.5) -- +(14.4,1.5)
 +(12,-.5) -- +(19.5,1.5)

 +(12,-.5) -- +(14.5,-2.5)
 +(12,-.5) -- +(19.5,-2.5)
 +(12,-.5) -- +(21.5,-2.5)

+(12,-4.5) -- +(14.5,-2.5)
 +(12,-4.5) -- +(19.5,-2.5)

 +(12,-4.5) -- +(14.5,-6.5)
 +(12,-4.5) -- +(19.5,-6.5)
 +(12,-4.5) -- +(21.5,-6.5)

 +(12,-8.5) -- +(14.5,-6.5)
 +(12,-8.5) -- +(19.5,-6.5)

+(14.5,1.5) node{}    +(19.5,1.5) node{}
+(14.5,-2.5) node{}   +(19.5,-2.5) node{}
+(14.5,-6.5) node{}   +(19.5,-6.5) node{}
+(21.5,1.5) node{}
+(21.5,-2.5) node{}
+(21.5,-6.5) node{}

+(16.75,1.5) node[rectangle, draw=white!0, fill=white!100]{ ${\textbf{.}}$}
+(17,1.5) node[rectangle, draw=white!0, fill=white!100]{ ${\textbf{.}}$}
+(17.25,1.5) node[rectangle, draw=white!0, fill=white!100]{ ${\textbf{.}}$}

+(16.75,-2.5) node[rectangle, draw=white!0, fill=white!100]{ ${\textbf{.}}$}
+(17,-2.5) node[rectangle, draw=white!0, fill=white!100]{ ${\textbf{.}}$}
+(17.25,-2.5) node[rectangle, draw=white!0, fill=white!100]{ ${\textbf{.}}$}

+(16.75,-6.5) node[rectangle, draw=white!0, fill=white!100]{ ${\textbf{.}}$}
+(17,-6.5) node[rectangle, draw=white!0, fill=white!100]{ ${\textbf{.}}$}
+(17.25,-6.5) node[rectangle, draw=white!0, fill=white!100]{ ${\textbf{.}}$}

+(12.5,1.5) -- +(16.2,1.5)
+(17.8,1.5) -- +(21.5,1.5)

+(12.5,-2.5) -- +(16.2,-2.5)
+(17.8,-2.5) -- +(21.5,-2.5)

+(12.5,-6.5) -- +(16.2,-6.5)
+(17.8,-6.5) -- +(21.5,-6.5)

+(12,3.5) -- +(21.5,1.5)
+(12,-.5) -- +(21.5,-2.5)
+(12,-4.5) -- +(21.5,-6.5)

(12,-.5) .. controls(16.75,-.5) .. (21.5,1.5)
(12,-4.5) .. controls(16.75,-4.5) .. (21.5,-2.5)
(12,-8.5) .. controls(16.75,-8.5) .. (21.5,-6.5)

(12.5,1.5) .. controls(17, 0) .. (21.5, 1.5)
(12.5,-2.5) .. controls(17, -4) .. (21.5, -2.5)
(12.5,-6.5) .. controls(17, -8) .. (21.5, -6.5)

 (4, 2) .. controls(2, -2.5) ..  (4, -7)
};
\end{tikzpicture}
\end{center}
\caption{The  graph $G=P_4$ and $H$.}\label{fig:g1-g2-g3}
\end{figure}
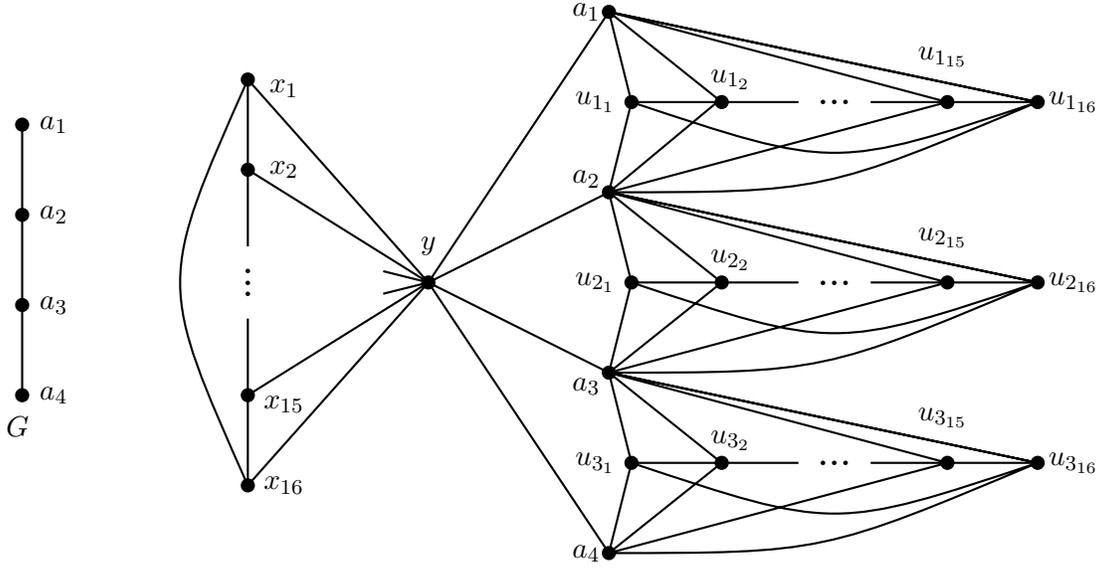

has three edges $e_1, e_2, e_3$,
$$H[\{x_i: 1\le i\le 16\}] \cong H[\{u_{1_i}: 1\le i\le 16\}] \cong H[\{u_{2_i}: 1\le i\le 16\}]\cong H[\{u_{3_i}: 1\le i\le 16\}]\cong C_{16}$$
$y$ is adjacent to $x_i$ for $1\le i \le 16$ and $a_l$ for  $1\le l\le 4$;  $u_{j_i}$ is adjacent to both $a_j$ and $a_{j+1}$ for $1\le j\le 3$ and
$1\le i\le 16$.

We claim that $G$ has a vertex cover of size at most $k$ if and only if $H$ has an RDRDF with weight at most $3k+3$. Hence the $NP$-completeness of the
restrained double Roman domination problem in general
graphs will be equivalent  to the $NP$-completeness of vertex cover problem. First, if $G$ has a vertex cover $C$ of size at most $k$, then the
function $f$ defined on $V(G)$ by $f(v) = 3$ for
$v \in C \cup \{y\}$ and $f(v) = 0$ otherwise,
is an RDRDF with weight at most $3k + 3$. On the other hand, suppose that $g$ is an RDRDF on $H$ with weight at most $3k + 3$. If $g(y)\ne 3$, then
there exist two cases.

Case 1. Let $g(y)\in \{0,1\}$. Then
$$\sum_{i=1}^{m}g(x_i)\ge \gamma_{rdR}(C_m) \ge \gamma_{dR}(C_m)\ge m >3|V(G)|+3 \ge 3k+3$$ that is a contradiction.

Case 2. Let $g(y)=2$ and $C_m=\{x_ix_{i+1}: (\mod\ m)\ 1\le i\le m\}$. Then $g(C_m)\ge 2m/3$ and $g(H)\ge 2m/3 +2k+2 = 2(3|V (G)| + 4)/3 +2k+2\ge
4k+14/3> 3k+3$ which is a contradiction.  Thus $g(y) = 3$. Similarly, we have $g(u) = 3$ or $g(v) = 3$
for any $e = uv \in  E(G)$. Therefore $C = \{v \in V : g(v) = 3\}$ is a vertex cover of $G$ and
$3|C| + 3 \le w(g) \le 3k + 3$. Consequently, $|C| \le k$.
\end{proof}

\section{$RDRD$ number of some graphs}
In this section we investigate the exact value of the restrained double Roman domination number of some graphs.
\begin{observation}\label{the-com-par} For complete graph $K_n$ and complete bipartite graph $K_{m,n}$,\\

\emph{(i)}   $\gamma_{rdR}(K_n)=3$ for $n\ge 2$.\\

\emph{(ii)} $\gamma_{rdR}(K_{n,m})=6$ for $m,n\ge 2$.\\

\emph{(iii)}  $\gamma_{rdR}(K_{1,m})=m+2.$

\emph{(iv)}  $\gamma_{rdR}(K_{n_1,n_2,\cdots, n_m})=\left\{
      \begin{array}{ll}
        3, &  \mbox{if}\ \emph{min}\{n_1 ,n_2,\cdots, n_m\}=1,\\
        6, & \hbox{otherwise.}
      \end{array}
    \right.$\\.
\end{observation}

\begin{theorem}\label{the-path}
For a path $P_n$ $(n\geq 4)$, $\gamma_{rdR}(P_n)=n+2$.\\
\end{theorem}
\begin{proof}
Assume that $n\ge 4$ and  $P_n=v_1v_2\cdots v_n$.  Define $h:V(P_n) \to \{0,1,2,3\}$ by  $h(v_{3i+2})=3$ for $0\le i \le n/3-1,\  h(v_{1})=h(v_n)=1$
and $h(v)=0$ otherwise, whenever
$n \equiv 0 \,({\rm mod}\, 3)$.\\
Define $h:V(P_n) \to \{0,1,2,3\}$ by  $h(v_{3i+1})=3$ for $0\le i \le (n-1)/3$  and $h(v)=0$ otherwise, whenever $n \equiv 1\, ({\rm mod}\, 3)$. \\
Define $h:V(P_n) \to \{0,1,2,3\}$ by  $h(v_{3i+2})=3$ for $0\le i \le (n-2)/3,\
h(v_{1})=1$  and $h(v)=0$ otherwise, whenever $n \equiv 2\, ({\rm mod}\, 3)$. Therefore $\gamma_{rdR}(P_n)\le n+2$ for $n\ge 4$.

Now we prove the inverse inequality.  It is straightforward to verify that $\gamma_{rdR}(P_n)=n+2$ for $4\le n\le 6$.  For $n\ge 7$ we proceed by
induction on $n$. Let $n\ge 7$ and
let the inverse inequality be true for every path of order less than $n$. Assume that  $f = (V_0, V_1, V_2, V_3)$ is a $\gamma_{rdR}$-function of $P_n$.
It is well known that $f(v_n)\ne 0$.
 If $f(v_n)=1$, then $f(v_{n-1}) \ge 2$.  Define $g: P_{n-1} \to \{0,1,2,3\}$, $g(v_i)=f(v_i)$ for $1\le i\le n-1$. Clearly, $g$ is an RDRD-function
 of $P_{n-1}$. It follows from the induction hypothesis that
$$\gamma_{rdR}(P_n)=w(f)=w(g)+1\ge \gamma_{rdR}(P_{n-1})+1\ge (n-1)+2 +1\ge n+2.$$
If $f(v_{n}) =2$, then  $f(v_{n-1})= 1$ and $f(v_{n-2})\ge 1$.  Define $g: P_{n-2} \to \{0,1,2,3\}$, $g(v_i)=f(v_i)$ for $1\le i\le n-2$. Clearly,
$g$ is a $RDRD$-function of $P_{n-2}$. As above we obtain,
$$\gamma_{rdR}(P_n)=w(f)=w(g)+3\ge \gamma_{rdR}(P_{n-2})+3\ge (n-2)+2 +3=n+3.$$
 If $f(v_{n}) =3$, then $f(v_{n-1})= 0$, $f(v_{n-2})= 0$ and $f(v_{n-3})= 3$. Define $g: P_{n-3} \to \{0,1,2,3\}$, $g(v_i)=f(v_i)$ for $1\le i\le n-3$.
 Clearly, $g$ is a $RDRD$-function
of $P_{n-3}$.  It also follows from the induction hypothesis that $$\gamma_{rdR}(P_n)=w(f)=w(g)+3\ge \gamma_{rdR}(P_{n-3})+3\ge (n-3)+2 +3=n+2.$$
 Thus the proof is complete.\\
\end{proof}

 \begin{theorem}\label{the-cycle}
For a cycle $C_n$, $(n\ge 3)$,
 $\gamma_{rdR}(C_n)=\left\{
      \begin{array}{ll}
        n, &  \mbox{if}\ n \equiv 0\ (\mbox{mod}\ 3), \\
        n+2, & \hbox{otherwise.}
      \end{array}
    \right.$\\
\end{theorem}

 \begin{proof}
 Assume that $n\ge 3$ and  $C_n=v_1v_2\cdots v_nv_1$.  Define $h:V(C_n) \to \{0,1,2,3\}$ by  $h(v_{3i})=3$ for $1\le i \le n/3$  and
 $h(v)=0$ otherwise, whenever  $n \equiv 0\, ({\rm mod}\, 3)$.\\
Define $h:V(C_n) \to \{0,1,2,3\}$ by  $h(v_{3i+1})=3$ for $0\le i \le (n-1)/3$  and $h(v)=0$ otherwise, whenever $n \equiv 1\, ({\rm mod}\, 3)$. \\
Define $h:V(C_n) \to \{0,1,2,3\}$ by  $h(v_{3i+2})=3$ for $0\le i \le (n-2)/3,\ h(v_{1})=1$  and $h(v)=0$ otherwise, whenever
$n \equiv 2\, ({\rm mod}\, 3)$. Therefore
$$\gamma_{rdR}(C_n)\le \left\{
      \begin{array}{ll}
        n, &  \mbox{if}\ n \equiv 0\ (\mbox{mod}\ 3), \\
        n+2, & \hbox{otherwise.}
      \end{array}
    \right.$$

 Now we prove the inverse inequality. For $n \equiv 0 \,({\rm mod}\, 3)$, since $\gamma_{rdR}(C_n)\ge \gamma_{dR}(C_n)=n$, (see \cite{al, bhh}),
 clearly the result holds.
 Let $n \not \equiv 0\  (\mbox{mod}\ 3)$ and let $f = (V_0, V_1, V_2, V_3)$ be a $\gamma_{rdR}$-function of $C_n$. Since the neighbor
 of vertex of weight $0$ is a vertex of weight $3$ and a vertex of weight $0$,
if $n \not \equiv 0\  (\mbox{mod}\ 3)$, there are two adjacent vertices $v_i, v_{i+1}$ in $C_n$ such that their weights are positive.
Now, if $f(v_i)\ge 2$ and $f(v_{i+1})\ge 2$, then by removing the edge $v_iv_{i+1}$, the resulted graph is $P_n$. Define
$g: P_{n} \to \{0,1,2,3\}$, $g(v_i)=f(v_i)$ for $1\le i\le n$. Clearly,
$g$ is an RDRD-function of $P_{n}$ with $w(g)=w(f)$. Since $w(g)\ge n+2$ then $w(f)\ge n+2$.\\
Let $f(v_i)\ge 2$ and $f(v_{i+1})= 1$. Then $f(v_{i+2})\ge  1$. Now  remove the edge $v_{i+1}v_{i+2}$ and obtain a $P_n$.  Define
$g: P_{n} \to \{0,1,2,3\}$, $g(v_i)=f(v_i)$ for $1\le i\le n$.
Clearly, $g$ is an RDRD-function of $P_{n}$ with $w(g)=w(f)$. Thus $w(f)\ge n+2$. \\
Let $f(v_i)=f(v_{i+1})= 1$. As above, we remove the edge $v_iv_{i+1}$ and the resulted graph $P_n$ has an RDRD-function $g$ of weight at least $w(f)$.
That is  $w(f)\ge n+2$.  Therefore the proof is complete.
\end{proof}

\section{Upper bounds on the $RDRD$ number}
In this section we obtain sharp upper bounds on  the restrained double Roman domination number of a graph.
\begin{proposition}\label{2n-1} Let $G$ be a  connected graph  of order $n\ge 2$. Then
$\gamma_{rdR}(G) \le  2n-1$, with equality if and only if $n=2$.
\end{proposition}
\begin{proof} If $w$ is a vertex of $G$, then define the fuction $f$ by $f(w)=1$ and $f(x)=2$ for $x\in V(G)\setminus\{w\}$.
Since $G$ is connected of order $n\ge 2$, we observe that $f$ is an RDRD function of $G$ of weight $2n-1$ and thus $\gamma_{rdR}(G) \le  2n-1$.
If $n\ge 3$, then $G$ contains a vertex $w$ with at least two neighbors $u$ and $v$. Now define the function $g$ by $g(u)=g(v)=1$ and $g(x)=2$ for $x\in V(G)\setminus\{u,v\}$.
Then $g$ is an RDRD function of $G$ of weight $2n-2$ and so $\gamma_{rdR}(G) \le  2n-2$ in this case. Since $\gamma_{rdR}(K_2)=3=2\cdot 2-1$, the proof is complete.
\end{proof}

\begin{proposition}\label{diam} Let $G$ be a  connected graph  of order $n\ge 2$. Then
$\gamma_{rdR}(G) \le  2n+1 - diam(G)$ and this bound is sharp for the path $P_n$ ($n\ge 4$).
\end{proposition}

\begin{proof} By Theorem \ref{the-path}, $\gamma_{rdR}(P_n) \le n+2$. Let  $P=v_1v_2\cdots v_{diam(G)+1}$  be a diametrical path in $G$. Let $g$
be a $\gamma_{rdR}$-function of $P$. Then $w(g)\le diam(G)+3$. Now we define an RDRD-function $f$ as:\\
$$f(x)=\left\{
      \begin{array}{ll}
        2, & x \notin V(P),\\
        g(x), & \hbox{otherwise.}
      \end{array}
    \right.$$

 It is clear that $f$ is an RDRD-function of $G$ of weight $w(f) \le 2(n-(diam(G)+1)) + diam (G)+3$. Therefore $\gamma_{rdR}(G) \le  2n+1 - diam(G)$.\\
 Theorem \ref{the-path} shows the sharpness of this bound.
\end{proof}

\begin{proposition} Let $G$ be a  connected graph  of order $n$ and circumference $c(G)<\infty$. Then
$\gamma_{rdR}(G) \le  2n +2 - c(G)$, and this bound is sharp for each cycle $C_n$ with $3 \nmid n$.
\end{proposition}

\begin{proof} Let $C$ be a longest cycle of $G$, that means $|V(C)|=c(G)$. By Theorem \ref{the-cycle}, $\gamma_{rdR}(C) \le c(G)+2$. Let $h$
be a $\gamma_{rdR}$-function on $C$. Then $w(h)\le c(G)+2$. Now we define an RDRD-function $f$ as:\\
$$f(x)=\left\{
      \begin{array}{ll}
        2, & x\notin V(C), \\
        h(x), & \hbox{otherwise.}
      \end{array}
    \right.$$\\

 It is clear that $f$ is an RDRD-function of $G$ of weight $w(f) \le 2(n-c(G)) + c(G)+2$. Therefore $\gamma_{rdR}(G) \le  2n+2 - c(G)$.\\
 For sharpness, if $G=C_n$ and $3\nmid n$, then $\gamma_{rdR}(C_n)=n+2= 2n+2 - n=2n+2-c(G)$.
\end{proof}

\begin{observation}\label{1}
Let $G$ be a graph and $f=(V_0,V_1,V_2)$ a $\gamma_{rR}$-function of $G$. Then $\gamma_{rdR}(G)\leq 2|V_1|+3|V_2|$.
\end{observation}

\begin{proof}
Let $G$ be a graph and $f=(V_0,V_1,V_2)$ a $\gamma_{rR}$-function of $G$. We define a function $g=(V_0',V_2',V_3')$ as follows:
$V_0'=V_0$, $V_2'=V_1$, $V_3'=V_2$. Note that under $g$, every vertex with a label $0$ has a neighbor assigned $3$ and each vertex with
label $1$ becomes a vertex with label $2$ and also $G[V_0']$ has no isolated vertex. Hence, $g$ is a  restrained double Roman dominating function.
Thus, $\gamma_{rdR}(G)\leq 2|V_2'|+3|V_3'|=2|V_1|+3|V_2|$.
\end{proof}

Clearly, the bound of observation \ref{1} is sharp, as can be seen with the path $G=P_4$, where $\gamma_{rR}(G)=4$ and $\gamma_{rdR}(G)=6$.
We also note that strict inequality in the bound can be achieved by the subdivided star $G=S(K_{1,k})$ which formed by subdividing each edge of
the star $K_{1,k}$, for $k\geq 3$, exactly once. Then it is simple to check that $\gamma_{rR}(G)=2k+1$ and $\gamma_{rdR}(G)=3k$. Hence, $|V_1|=1$
and $|V_2|=k$, and so, $3k=\gamma_{rdR}(G)<2|V_1|+3|V_2|=2+3k.$

\begin{lemma}\label{lem1} If a graph $G$ has a non-pendant edge, then there  is a $\gamma_{rdR}(G)$-function $f = (V_0, V_1, V_2, V_3)$ such that
$V_0\cup V_1 \ne \emptyset$.
\end{lemma}

\begin{proof}
 If $\gamma_{rdR}(G)<2n$, then obviously $V_0\cup V_1 \ne \emptyset$. Now we show that $\gamma_{rdR}(G)<2n$.
Let $uw$ be a non-pendant edge with $\deg(u)$ and $\deg(w)$ be at least $2$.\\
Assume that $N_{G}(u)\cap N_{G}(w)\ne \emptyset$, and let $v$ be a  vertex in $N_{G}(u)\cap N_{G}(w)$. Then the function
$f=(V_0 =\{u,w\}, V_1 =\emptyset, V_2= V(G) \setminus \{u,w,v\}, V_3=\{v\})$
is an RDRD-function of $G$ with $w(f)\le 2n-3$.\\
Assume that  $N_{G}(u)\cap N_{G}(w)= \emptyset$, and let $a\in N_G(u)\setminus \{w\}$ and $b\in N_G(w)\setminus \{u\}$.
Then the function $f=(V_0= \{u,w\}, V_1=\emptyset, V_2= V(G) \setminus \{u,w,a,b\}, V_3=\{a,b\})$ is an RDRD-function of $G$
with $w(f)\le 2n-2$. This completes the proof.
All in all the proof is complete.
\end{proof}
\vspace{3mm}

\begin{figure}[h]
\tikzstyle{every node}=[circle, draw, fill=black!, inner sep=0pt,minimum width=.16cm]
\begin{center}
\begin{tikzpicture}[thick,scale=.6]
  \draw(0,0) { 
 +(-1,1) node{}
    +(1,1) node{}
 +(0,-1) node{}

 +(-1,1) -- +(1,1)
 +(0,-1) -- +(-1,1)
 +(0,-1) -- +(1,1)


 +(-1,-3) node{}
    +(1,-3) node{}
 +(0,-1) node{}

 +(-1,-3) -- +(1,-3)
 +(0,-1) -- +(-1,-3)
 +(0,-1) -- +(1,-3)


 +(-2,0) node{}
    +(-2,-2) node{}
 +(0,-1) node{}

 +(-2,0) -- +(-2,-2)
 +(0,-1) -- +(-2,-0)
 +(0,-1) -- +(-2,-2)


 +(2,0) node{}
    +(2,-2) node{}
 +(0,-1) node{}

 +(2,0) -- +(2,-2)
 +(0,-1) -- +(2,0)
 +(0,-1) -- +(2,-2)

 +(2,-3) node{}

  +(0,-1) -- +(2,-3)
  +(0.3,-3.8) node[rectangle, draw=white!0, fill=white!100]{ $\small{H_{10}}$}

 +(6,1) node{}
    +(8,1) node{}
 +(7,-1) node{}

 +(6,1) -- +(8,1)
 +(7,-1) -- +(6,1)
 +(7,-1) -- +(8,1)


 +(6,-3) node{}
    +(8,-3) node{}
 +(7,-1) node{}

 +(6,-3) -- +(8,-3)
 +(7,-1) -- +(6,-3)
 +(7,-1) -- +(8,-3)


 +(5,0) node{}
    +(5,-2) node{}
 +(7,-1) node{}

 +(5,0) -- +(5,-2)
 +(7,-1) -- +(5,-0)
 +(7,-1) -- +(5,-2)


 +(9,0) node{}
    +(9,-2) node{}
 +(7,-1) node{}

 +(9,0) -- +(9,-2)
 +(7,-1) -- +(9,0)
 +(7,-1) -- +(9,-2)
+(7.2,-3.8) node[rectangle, draw=white!0, fill=white!100]{ $\small{F_{9}}$}

 };
\end{tikzpicture}
\end{center}
\caption{The  graph $H_{10},\ F_{9}$ .}\label{fig:g1-g2-g3}
\end{figure}
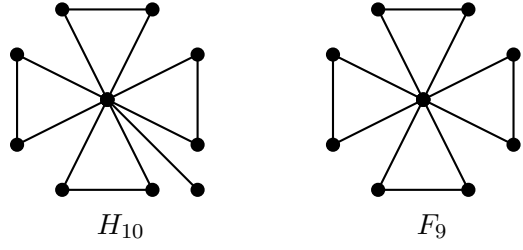

For any integer  $n \ge 3$, let $H_n$  be the graph obtained from $(n-2)/2$ copies
of $K_2$ and a copy of $K_1$ by adding a new vertex and joining it to both leaves of each $K_2$ and the given $K_1$, and
let $F_n$  be the graph obtained from $(n-2)/2$ copies of $K_2$ by adding a new vertex and joining it to both leaves of each $K_2$. Thus
 for $n \ge 4$, $H_n$ have a vertex of degree $n-1$, a vertex of degree $1$  and  other vertices
of degree two and for $n \ge 3$, $F_n$ have a vertex of degree $n-1$ and  other vertices
of degree two. Figure 2 shows the graph $H_{10}$ and $F_9$. Let $\mathcal{H} = \{H_n :n \ge 4\ \mbox{is\ even}\}$ and
$\mathcal{F} = \{F_n :n \ge 3\ \mbox{is\ odd}\}$.

\begin{theorem}\label{the} For every connected graph $G$ of order $n \ge 3$ with $m$ edges,
$\gamma_{rdR}(G) \ge 2n + 1- \lceil(4m-1)/3\rceil$, with equality if and only if
$G \in  \mathcal{H}\cup \mathcal{F}$  or $G\in\{K_{1,2},K_{1,3},K_{1,4}\}$.
\end{theorem}

\begin{proof}
 If $G=K_{1,n-1}$ is a star, then $\gamma_{rdR}(G)=n+1$ and $m=n-1$. Now it is easy to see that
 $\gamma_{rdR}(K_{1,n-1})= 2n + 1- \lceil(4m-1)/3\rceil$ for $3\le n\le 5$ and
$\gamma_{rdR}(K_{1,n-1})>2n + 1- \lceil(4m-1)/3\rceil$ for $n\ge 6$.
Next assume that  $G$ is not a star.  By Lemma \ref{lem1} there is a $\gamma_{rdR}(G)$-function of $f= (V_0, V_1, V_2, V_3)$  such that
$V_0\cup V_1\ne \emptyset$. It is well known that,
 the induced subgraph $G[V_0]$ has no isolated vertex. Therefore, $|E(G[V_0])| \ge |V_0|/2$. Let $V'_0=\{v\in V_0: N(v) \subseteq V_2\}$ and
 $V''_0=\{v\in V_0: v \,\,{has\,\,a\,\,neighbor\,\,in}\,\,V_3\}$. Then $|E(V_0,V_2)| \ge 2|V'_0|$, $|E(V_0,V_3)| \ge |V''_0|$ and
 $|E(V_1,V_2\cup V_3)| \ge |V_1|$.
 Therefore

 $$|E(G)|=m\ge |V_0|/2+ 2|V'_0|+ |V''_0|+ |V_1|.$$
 Since $|V_0|=|V'_0|+|V''_0|$, we deduce that
 \begin{equation}\label{EQ11}
 (4m-1)/3\ge 2|V_0|+ 4/3|V'_0|+ 4/3|V_1|-1/3
 \end{equation}

  and thus
  \begin{equation}\label{EQ12}
 2n+1 - \lceil(4m-1)/3\rceil \le 2n+1 - (4m-1)/3 \le 2n+1 -2|V_0|-4/3|V'_0|- 4/3|V_1|+1/3.
 \end{equation}

Since $\gamma_{rdR}(G) = |V_1|+2|V_2|+3|V_3|$, $|V_0|+|V_1|+|V_2|+|V_3|=n$ and $2n+1=2|V_0|+2|V_1|+2|V_2|+2|V_3|+1$, we obtain
\begin{eqnarray*}
2n+1 -2|V_0|-4/3|V'_0|- 4/3|V_1|+1/3 & = & -4/3|V'_0|+ 2/3|V_1|+2|V_2|+2|V_3|+4/3\\
& = & \gamma_{rdR}(G)-4/3|V'_0|-1/3|V_1|-|V_3|+4/3.
\end{eqnarray*}
 Next we will show that
\begin{equation}\label{EQ13}
 \gamma_{rdR}(G)-4/3|V'_0|-1/3|V_1|-|V_3|+4/3\le \gamma_{rdR}(G)
\end{equation}
or $\gamma_{rdR}(G) \ge 2n + 1- \lceil(4m-1)/3\rceil$.
If  $|V'_0|\ge 1$, then $-4/3|V'_0|-1/3|V_1|-|V_3|+4/3\le 0$ and so $\gamma_{rdR}(G)-4/3|V'_0|-1/3|V_1|-|V_3|+4/3\le \gamma_{rdR}(G)$.\\
Let now  $|V'_0|=0$. Note that the condition $V_0\cup V_1\ne \emptyset$ implies $V''_0 \cup V_1\ne \emptyset$.\\

Assume next that $V_1= \emptyset$. We deduce that  $|V''_0|\ge 1$ and therefore $|V_3|\ge 1$. If there are at least two vertices of weight $3$,
then $\gamma_{rdR}(G)-4/3|V'_0|-1/3|V_1|-|V_3|+4/3<\gamma_{rdR}(G)$.\\
If there is only one vertex of weight $3$, then $m\ge n-1+\frac{n-1}{2}=\frac{3(n-1)}{2}$. We deduce that
$\gamma_{rdR}(G)\ge 3  \ge 2n+1 -\left\lceil \frac{6(n-1)-1}{3}\right\rceil
\ge 2n+1 -\left\lceil \frac{4m-1}{3}\right\rceil$, with equality if and only if $|V_2|=0$, $n$ is odd and $m=\frac{3(n-1)}{2}$, that means
$G\in{\cal F}$.\\

Now assume that $|V_1|\ge 1$.  If $|V''_0|\ge 1$, then  $|V_3|\ge 1$ and thus $\gamma_{rdR}(G)-4/3|V'_0|-1/3|V_1|-|V_3|+4/3\le \gamma_{rdR}(G)$.
Next let $|V''_0|=0$. If  $|V_3|\ge 1$, then $\gamma_{rdR}(G)-4/3|V'_0|-1/3|V_1|-|V_3|+4/3\le \gamma_{rdR}(G)$. Now assume that $|V_3|=0$.
This implies that all vertices have weight $1$ or $2$. If $3\le n\le 5$, then it is easy to see that
$\gamma_{rdR}(G)> 2n+1-\left\lceil\frac{4m-1}{3}\right\rceil$.
Let now $n\ge 6$. If $|V_1|\ge 5$, then $\gamma_{rdR}(G)-4/3|V'_0|-1/3|V_1|-|V_3|+4/3<\gamma_{rdR}(G)$. Otherwise $|V_1|\le 4$, $|V_2|\ge n-4$ and
$m\ge n-1$. This implies
$$\gamma_{rdR}(G)\ge 2(n-4)+4=2n-4>2n+1-\left\lceil\frac{4(n-1)-1}{3}\right\rceil\ge 2n+1-\left\lceil\frac{4m-1}{3}\right\rceil.$$
 \\.

Thus $\gamma_{rdR}(G) \ge 2n+1 -(4m-1)/3\ge 2n+1 -\lceil(4m-1)/3\rceil$.\\
For equality: If $G\in \mathcal{H}$, then $G=H_n$ for $n\ge 4$ even and $|E(H_n)|=3(n-2)/2+1$.
Thus $2n+1- (4(3(n-2)/2+1)-1)/3= 2n+1- \lceil (4(3(n-2)/2+1)-1)/3 \rceil = 2n+1 -2(n-2)-1=4= \gamma_{rdR}(H_n)$.
If $G\in \mathcal{F}$, then $G=F_n$ for $n\ge 3$ odd and $|E(F_n)|=3(n-1)/2$.
Thus $2n+1- \lceil(4(3(n-1)/2))-1/3\rceil = 2n+1 -2(n-1)=3= \gamma_{rdR}(F_n)$.

Conversely, assume that $\gamma_{rdR}(G)=2n+1-\lceil(4m-1)/3\rceil$. Then all inequalities   occurring in the proof become equalities.
In the case  $|V_1|=0$, we have seen above that we have equality if and only if $G\in{\cal F}$. In the case $|V_1|\ge 1$, we have seen above
that $|V_3|\ge 1$. Therefore the equality in Inequality
(\ref{EQ13}) leads to $|V_3|=|V_1|=1$ and $|V'_0|=0$. Hence $V_0=V''_0$.  Thus equality in Inquality (\ref{EQ11}) or equivalently, in the
inequality $|E(G)|=m\ge |V_0|/2+ 2|V'_0|+ |V''_0|+ |V_1|$
leads  to $m=3/2|V''_0|+1$. Now let the vertices $v, u$ be of weight $3,1$ respectively.
Then $m=|E(G)| \ge |E(v, V''_0)| + G[V''_0] +1 \ge |V''_0|+ 1/2|V''_0|+1=3/2|V''_0|+1$.
If $|V_2|\ne 0$, then the connectivity of $G$ leads to the contradiction $m\ge 3/2|V''_0|+2$. Consequently, $|V_2|=0, |V_0|=(2m-2)/3$ and $u$ and $v$
are adjacent. Since $G$ is connected, $G\in \mathcal{H}$.
\\
\end{proof}

\section{$RDRD$-set versus $RRD$-set}
One of the aim of  studying these parameters is that to see the related  between them and compare each together.

\begin{proposition}\label{cor1}
For any graph $G$, $\gamma_{rdR}(G)\leq 2\gamma_{rR}(G)$ with equality if and only if $G=\overline{K_n}$.
\end{proposition}

\begin{proof}
Let $f=(V_0,V_1,V_2)$ be a $\gamma_{rR}$-function of $G$. Since $\gamma_{rR}(G)=|V_1|+2|V_2|$, by Observation \ref{1}, we have that
$\gamma_{rdR}(G)\leq 2|V_1|+3|V_2|=\gamma_{rR}(G)+|V_1|+|V_2|\leq 2\gamma_{rR}(G)$.
If $\gamma_{rdR}(G)=2\gamma_{rR}(G)=2|V_1|+4|V_2|$, then since $\gamma_{rdR}(G)\leq 2|V_1|+3|V_2|$, we must have $V_2=\emptyset$. Hence,
$V_0=\emptyset$ must hold, and so $V=V_1$. By definition of $\gamma_{rR}$-function, we deduce that no two vertices in $G$ are adjacent, for otherwise,
if $u$ and $v$ are adjacent, then only one of them in every $\gamma_{rdR}$-function on $G$ has a label of $2$ which contradicts with
$\gamma_{rdR}(G)=2\gamma_{rR}(G)$.
\end{proof}


The proof of Lemma \ref{lem1} shows the next proposition.

\begin{proposition} If $G$ contains a triangle, then $\gamma_{rdR}(G) \le 2n-3$.
\end{proposition}
\begin{theorem}\label{the-111}
For every graph $G$, $\gamma_{rR}(G) < \gamma_{rdR}(G$).
\end{theorem}
\begin{proof}
Let $f=(V_0,V_1,V_2,V_3)$ be a $\gamma_{rdR}(G)$-function. If $V_3 \ne \emptyset$, then $(V'_0=V_0, V'_1=V_1 , V'_2=V_2\cup V_3)$ is an
RRD-function $g$ such that $w(g)< w(f)$. Let $V_3=\emptyset$. If $V_0=\emptyset$, then, since $V_2 \ne \emptyset$,
$g=(\emptyset, V'_1=V_1\cup V_2, \emptyset)$ is an RRD-function such that $w(g)< w(f)$. If $V_0 \ne \emptyset$, then $|V_2|\ge 2$. Let $f(v)=2$
for a vertex $v$. Then $g=(V'_0=V_0, V'_1=V_1\cup \{v\}, V'_2=V_2-\{v\})$ is an RRD-function $g$ for which $w(g)< w(f)$. Therefore
$\gamma_{rR}(G) < \gamma_{rdR}(G)$.
\end{proof}
\vspace{2mm}

As an immediate consequence of Proposition \ref{cor1}, we have.
\begin{corollary}
For any nontrivial connected graph $G$, $\gamma_{rdR}(G) < 2\gamma_{rR}(G)$.
\end{corollary}

\begin{theorem}\label{the-222} Let  $G$ be a graph of order $n$. Then $\gamma_{rdR}(G)=\gamma_{rR}(G)+1$ if and only if $G$ is one of the following
graphs.\\
\emph{1}. $G$ has a vertex of degree $n-1$.\\
\emph{2}. There exists a  subset $S$ of $V(G)$ such that:\\
\emph{2.1}. every vertex of $V-S$ is adjacent to a vertex in $S$,\\
\emph{2.2}. there are two subsets $A_0$ and $A_1$ of $V-S$ with $A_0\cup A_1=V-S$ such that $A_0$ is the set of non-isolated
vertices in $N(S)$ and each vertex in $A_0$ has at least two neighbors in $S$,\\
\emph{2.3}. for any $2$-subset $\{a,b\}$ of $S$, $N(\{a,b\})\cup A_0 \ne \emptyset$ and for a $3$-subset $\{x,y,z\}$ of $S$,
if $\{x,y,z\} \cap A_0 \ne \emptyset$, then there are three vertices $u,v,w$ in $A_0$ such that $N(u)\cup S=\{x,y\}$, $N(v)\cup S=\{x,z\}$ and
$N(w)\cup S=\{y,z\}$.
\end{theorem}

\begin{proof} Let $\gamma_{rdR}(G)=\gamma_{rR}(G)+1$ with a $\gamma_{rdR}(G)$-function $f=(V_0,V_1,V_2,V_3)$ and a $\gamma_{rR}(G)$-function
$g=(U_0,U_1,U_2)$.
If $V_3\ne \emptyset$, then $|V_3|=1$.  Because if $|V_3|\ge 2$ then by  changing $3$ to $2$ we obtain a RRD-function $h$ with $w(h)< w(g)$, a
contradiction. Let $V_3=\{v\}$.
 In addition, we note that $|V_2|=0$. I we suppose that $|V_2|\ge 1$, then let $u\in V_2$. Then
$h=(V'_0=V_0, V'_1=V_1\cup \{u\}, V'_2=V_2\cup \{v\})$ is an RRD-function $g$ for which $w(h)< w(g)$, a contradiction.
Thus all vertices different from $v$ are adjacent to the vertex $v$ such that the non-isolated vertices in $N(v)$ are assigned with $0$ and the isolated
vertices in $N(v)$ are assigned with $1$.
In this case $U_0=V_0, U_1=V_1$ and $U_2=V_3$.\\

If $V_3= \emptyset$, then $V_2\ne \emptyset$ and $|V_2| \ge 2$.  In this case, there must exist a vertex $v\in V_2$ such that
$U_0= V_0, U_1=V_1\cup \{v\}$ and $U_2=V_2-\{v\}$. There is such a function $f$ if we guarantee a subset $S$ of $V(G)$ with each vertex of
weight $2$ for which every other vertex in $V-S$ has to adjacent to a vertex of $S$, that is the condition 2.1 holds.\\ Since we can only
change one of vertices of weight $2$ in $f$ to a vertex of weight $1$ in $g$,
there must be existed two subsets $A_0$ and $A_1$ in $V-S$  such that the conditions 2.2 and 2.3 hold.\\

Conversely, if the condition 1 holds, then $f=(V_0, V_1, \emptyset, V_3=\{v\})$ and $g=(U_0=V_0, U_1=V_1, U_2=\{v\})$ are
$\gamma_{rdR}(G)$-function and  $\gamma_{rR}(G)$-function respectively where $V_0$ is the set of non-isolated vertices in $N(v)$ and
$V_1$ is the set of isolated vertices in $N(v)$. Thus $\gamma_{rdR}(G)=\gamma_{rR}(G)+1$.\\
If the condition 2 holds, then we can have only one vertex of weight $2$ in $G$ under $f$ such that it changes to the weight $1$ in $G$ under $g$.
Thus $\gamma_{rdR}(G)=\gamma_{rR}(G)+1$.
\end{proof}

We showed that for any graph $G$, $\gamma_{rdR}(G)\le 2\gamma_{rR}(G)$ and the equality holds if and on if $G$ is a trivial graph $\overline{K_n}$.
Hence, for any nontrivial graph $G$, $\gamma_{rdR}(G)\le 2\gamma_{rR}(G)-1$. Now we characterise graph $G$ with this property
$\gamma_{rdR}(G)= 2\gamma_{rR}(G)-1$.

\begin{theorem}
If $G$ is a  nontrivial graph, then $\gamma_{rdR}(G)\le 2\gamma_{rR}(G)-1$. If $\gamma_{rdR}(G)=2\gamma_{rR}(G)-1$, then $G$ consists of a $K_2$ and $n-2$
isolated vertices or $G$ consists of a vertex $h$ and two disjoint vertex sets $H$ and $R$ such that $H=N(h)$, $G[H]$ does not have isolated vertices,
$G[R]$ is trivial, there is no edge between $h$ and $R$ and $N(h)\cap N(R)\neq N(h)$.
\end{theorem}

\begin{proof} Since $G$ is a nontrivial graph, Proposition \ref{cor1} implies $\gamma_{rdR}(G)\le 2\gamma_{rR}(G)-1$. Now we investigate the equality.\\
Let $\gamma_{rdR}(G)= 2\gamma_{rR}(G)-1$, where $f=(V_0,V_1,V_2,V_3)$ is a $\gamma_{rdR}(G)$-function and  $g=(U_0,U_1,U_2)$ is a
$\gamma_{rR}(G)$-function.
Then $2|U_1|+4|U_2|-1=|V_1|+2|V_2|+3|V_3|$. On the other hand,  since $2|U_1|+4|U_2|-1=|V_1|+2|V_2|+3|V_3|=\gamma_{rdR}(G) \le 2|U_1|+3|U_2|$,
it follows that $|U_2|\le 1$.

If $U_2=\emptyset$, then $|U_0|=0$ and therefore $|U_1|=n$. Using the inequality above, we obtain
$$2n-1=2|U_1|-1\le\gamma_{rdR}(G)\le 2|U_1|=2n.$$
If $\gamma_{rdR}(G)=2n$, then $G$ is trivial, a contradiction. If $\gamma_{rdR}(G)=2n-1$, then Proposition \ref{2n-1} shows that
$G$ consists of a $K_2$ and $n-2$ isolated vertices.

Let now $|U_2|= 1$ such that $U_2=\{h\}$, $H=N(h)$, $R=V(G)\setminus N[h]=\{u_1,u_1,\ldots,u_p\}$.
Clearly, $U_0\subseteq H$ and $R\subseteq U_1$.

If $H$ contains exactly $s\ge 1$ isolated vertices, then $\gamma_{rR}(G)=2+s+p$ and therefore
$\gamma_{rdR}(G)\le 3+s+2p\le 2\gamma_{rR}(G)-2$, a contradiction. Hence $H=N(h)$ does not contain isolated vertices and thus
$\gamma_{rR}(G)=p+2$.

If $G[R]$ contains an edge, then we obtain the contradiction  $\gamma_{rdR}(G)\le 3+2p-1=2p+2\le 2\gamma_{rR}(G)-2$. Thus $G[R]$ is trivial.

If there is an edge between $h$ and $R$, then we also obtain the contradiction  $\gamma_{rdR}(G)\le 3+2p-1=2p+2\le 2\gamma_{rR}(G)-2$.

If $N(h)\cap N(R)=N(h)$, then $f=(H,\emptyset,\{h\}\cup R,\emptyset)$ is an RDRD function of $G$, and hence
$\gamma_{rdR}(G)\le 2p+2\le 2\gamma_{rR}(G)-2$, a contradiction.
\end{proof}

\section{Trees}
In this section we study the restrained double Roman domination of trees.\\

\begin{theorem}\label{the-tree1}
If $T$ is a tree of order $n\geq2$, then $\gamma_{rdR}(T)\leq \lceil\frac{3n-1}{2}\rceil$. The equality holds if $T\in\{P_2,P_3,P_4,P_5, S_{1,2}, ws(1,n, n-1), ws(1,n, n-2)\}$.
\end{theorem}

\begin{proof}
 Let $T$ be a tree of order $n\geq2$. We will proceed by induction on $n$. If $n=2$, then $\gamma_{rdR}(T)=3=\lceil\frac{3n-1}{2}\rceil$. If $n\geq3$,
 then $diam(T)\geq2$.
If $diam(T)=2$, then $T$ is the star $K_{1,n-1}$ for $n\geq3$ and $\gamma_{rdR}(T)=n+1\leq \lceil\frac{3n-1}{2}\rceil$. If $diam(T)=3$, then $T$ is a
double star $S_{r,s}$ for $1\leq r\leq s$. Hence, $n=r+s+2\geq4$.
If $r=1=s$, then $T=P_4$ and $\gamma_{rdR}(T)=6\leq\lceil\frac{12-1}{2}\rceil$. If $r=1, s\ge 2$, then $n=s+3$ and
$\gamma_{rdR}(T)=s+5\leq\lceil\frac{3(s+3)-1}{2}\rceil$. If $r\ge 2, s\ge 2$,
then $n=r+s+2$ and  $\gamma_{rdR}(T)=r+s+4\leq\lceil\frac{3(r+s+2)-1}{2}\rceil$.\\
Hence, we may assume that $diam(T)\geq4$. This implies that $n\geq5$. Assume that any tree $T'$ with order $2\leq n'<n$ has
$\gamma_{rdR}(T')\leq \lceil\dfrac{3n'-1}{2}\rceil$. Among all longest paths in $T$,
choose $P$ to be one that maximizes the degree of its next-to-last vertex $v$, and let $w$ be a leaf neighbor of $v$. Note that by our
choice of $v$, every child of $v$ is a leaf. Since $deg(v)\geq2$, the vertex $v$
has at least one leaf as a child. Now we put $T'=T-T_v$ where the order of the substar $T_v$ is $k+1$ with $k\geq1$. Note that since
$diam(T)\geq4$, $T'$ has at least three vertices, that is, $n'\geq3$. Let $f'$ be a
$\gamma_{rdR}$-function of $T'$. Form $f$ from $f'$ by letting $f(x)=f'(x)$ for all $x\in V(T')$, $f(v)=2$, and $f(z)=1$ for all leaf neighbors of $v$.
Thus $f$ is a restrained double Roman dominating function of $T$,
implying that $\gamma_{rdR}(T)\leq \gamma_{rdR}(T')+k+2 \le \lceil\dfrac{3(n-k-1)-1}{2}\rceil+k+2=\lceil\dfrac{3n-k}{2}\rceil
\leq \lceil\dfrac{3n-1}{2}\rceil$.\\

If $T\in \{P_2,P_3,P_4,P_5, S_{1,2}, , ws(1,n, n-1), ws(1,n, n-2)\}$, then  clearly
$\gamma_{rdR}(T)=\lceil\dfrac{3n-1}{2}\rceil$.
\end{proof}
\vspace{2mm}

\begin{theorem}\label{the-tree2}
  For every tree $T$ of  order $n\geq 3$, with $l$ leaves and $s$ support vertices, we have $\gamma_{rdR}(T)\leq\dfrac{4n+2s-l}{3}$, and this
  bound is sharp for the family of stars ($K_{1,n-1}$ $n\geq 3$),
 double stars,  caterpillars for which each vertex is a leaf  or a support vertex and all support vertices have even degree
 $2m$ or at most two end support vertices has degree $2m-1$ and the other support vertices has degree $2m$, wounded spiders in which
 the central vertex is adjacent with at least two leaves.
\end{theorem}
\begin{proof}
Let $T$ be a tree with order $n\geq3$. Since $n\geq3$, $diam(T)\geq2$. If $diam(T)=2$, then $T$ is the star $K_{1,n-1}$ for $n\geq3$ and
$\gamma_{rdR}(T)=n+1\leq\dfrac{4n+2-(n-1)}{3}=\dfrac{3n+3}{3}=n+1$.
If $diam(T)=3$, then $T$ is a double star $S_{r,t}$ for $1\leq r\leq t$. We have $\gamma_{rdR}(T)=n+2=\dfrac{4n+2s-l}{3}$. Hence, we may assume
$diam(T)\geq4$. Thus, $n\geq5$. Assume that any tree
$T'$ with order $3\leq n'<n$, $l'$ leaves and $s'$ support vertices has $\gamma_{rdR}(T')\leq\dfrac{4n'+2s'-l'}{2}$. Among all longest paths in $T$,
choose $P$ to be one that maximizes the degree of its next-to-last vertex $u$,
and let $x$ be a leaf neighbor of $u$, $w$ be a parent vertex of $v$ and $v$ be a parent vertex of $u$. Note that by our choice of $u$, every child
of $u$ is a leaf. Since $t=deg(u)\geq2$, the vertex $u$ has at least one leaf children.
We now consider the two cases are as follows:\\
\textbf{Case 1}. $deg(v)\geq3$. In this case, we put $T'=T-T_u$, where the order of the star $T_u$ is $t$ with $t\geq2$. Note that since $diam(T)\geq4$,
$T'$ has at least three vertices, that is, $n'\geq3$. Let $f'$ be a
$\gamma_{rdR}$-function of $T'$. Thus we have $n'=n-t$, $l'=l-(t-1)$ and $s'=s-1$. Clearly,
$\gamma_{rdR}(T)\leq \gamma_{rdR}(T')+t+1\leq\dfrac{4(n-t)+2(s-1)-(l-(t-1))}{3}+t+1=\dfrac{4n+2s-l}{3}$.\\
 \textbf{Case 2}. $deg(v)=2$. We now consider the following two subcases.\\
\textbf{i}. $deg(w)>2$. Then we put $T'=T-T_v$ where order of subtree $T_v$ is $t+1$. Clearly, we have $n'=n-(t+1)$, $s'=s-1$ and $l'=l-(t-1)$. Thus,
$\gamma_{rdR}(T)\leq \gamma_{rdR}(T')+t+2\leq \dfrac{4(n-t-1)+2(s-1)-(l-(t-1))}{3}+t+2=\dfrac{4n+2s-l-1}{3}\leq \dfrac{4n+2s-l}{3}$.\\
\textbf{ii}. $deg(w)=2$. Then we put $T'=T-T_v$, where the order of the subtree $T_v$ is $t+1$. Thus in this case, $w$ in the subtree $T'$
becomes a leaf and  we have $n'=n-(t+1)$, $s'\le s$ and $l'= l-(t-1)+1$. Thus,
$\gamma_{rdR}(T)\leq \gamma_{rdR}(T')+t+2\leq \dfrac{4(n-t-1)+2(s)-(l-(t-1)+1)}{3}+t+2=\dfrac{4n+2s-l}{3}$.
\end{proof}

\begin{theorem}\label{the-tree3}
If $T$ is a  tree, then $\gamma_r(T)+1\le \gamma_{rdR}(T)\le 3\gamma_r(T)$, and equality for the lower bound holds if and only if $T$ is a star.
The upper bound is sharp for the paths $P_{m}$  ($m\equiv 1\ \mbox{mod}\ 3$),
The cycles $C_n$  ($n\equiv 0,\,1\ \mbox{mod}\ 3$), the complete graphs $K_n$,  the complete bipartite graphs
$K_{n,m}\ (m,n \ge 2)$ and the multipartite graphs $K_{n_1,n_2,\cdots, n_m},\ (m\ge 3)$.
\end{theorem}
\begin{proof}
Let $T$ be a tree. Since at least one vertex has value $2$ under  any $RDRD$ function of $T$, we see that $\gamma_r(T)+1 \le \gamma_{rdR}(T)$.
If we assign
the value $3$ to the vertices in a $\gamma_r(T)$-set, then  we obtain an RDRD function of $T$. Therefore $\gamma_{rdR}(T)\le 3\gamma_r(T)$.\\
The sharpness of the upper bound  is deuced from Propositions 1-7 of \cite{domke} and Observation \ref{the-com-par}, Theorem
\ref{the-path} and Theorem \ref{the-cycle}.\\
For equality of the lower bound, if $T=K_{1,n-1}$ is a star, then it is clear $\gamma_{rdR}(T)=n+1$ and $\gamma_{r}(T)=n$. If $T$ is a tree and
$\gamma_{rdR}(T)=\gamma_{r}(T)+1$, then we have only one vertex of value $2$ in any $\gamma_{rdR}(T)$-function and the other vertices of positive
weight have value $1$.
In addition, the vertices of value 1 are adjacent to the vertex of value 2, and therefore $T$ is a star.
least one vertex
\end{proof}

The following result gives us the RDRD of $G$  in terms of the size of $E(G)$, and order of $G$.
\begin{proposition}\label{prop-tree4}   Let $G$ be a connected graph  $G$  of order $n\ge 2$ with $m$ edges. Then
$\gamma_{rdR}(G) \le 4m-2n+3$, with equality if and only if $G$ is a tree with
$\gamma_{rdR}(G) = 2n-1$.
\end{proposition}
\begin{proof}
For the given connected graph, $m\ge n-1$ and  according to Proposition \ref{diam} $\gamma_{rdR}(G) \le 2n-1 =4n-4 -2n +3\le 4m-2n+3$.\\
If $\gamma_{rdR}(G) = 4m-2n+3$, then
$m=n-1$ and $G$ is a tree with $\gamma_{rdR}(G) = 2n-1$.\\
Conversely, assume that $G$  is a tree with $\gamma_{rdR}(G) = 2n-1$. Hence $\gamma_{rdR}(G) = 4m-2n+3$.
\end{proof}

\section{Conclusions and problems}
The concept of restrained double Roman domination in graphs was initially investigated in this paper. We studied the computational complexity of this
concept and proved some bounds on the $RDRD$ number of graphs. In the case of trees, we characterized all trees attaining the exhibited bound.  We now
conclude the paper  with some problems suggested by this research.\vspace{1mm}\\
$\bullet$ For any graph $G$, provided the characterizations of graphs with small or large $RDRD$ numbers.
$\bullet$ It is also worthwhile proving some other nontrivial sharp bounds on $\gamma_{rdR}(G)$ for general graphs $G$ or some well-known families
such as, chordal, planar, triangle-free, or claw-free graphs.\vspace{1mm}\\
$\bullet$ The decision problem RESTRAINED DOUBLE ROMAN DOMINATION is NP-complete  for general graphs, as proved in Theorem \ref{the-NP}. By the way,
there might be some families of graphs such that $RDRD$ is  NP-complete for them or there might be some  polynomial-time algorithms for computing
the $RDRD$ number of some well-known families of graphs, for instance, trees. Can you provide these families?\\
$\bullet$ In Theorems \ref{the-tree1} and \ref{the-tree2} we showed upper bounds for $\gamma_{rdR}(T)$. The sufficient and necessity conditions for equality may be problems.

\end{document}